%% file: main.tex
\documentclass{gtpart}
\title{Images of dominant endomorphisms of affine space}

\author{Viktor Balch Barth} 
\email{viktorbb@math.uio.no}
\address{Department of Mathematics, UiO, Norway.}

\author{Tuyen Trung Truong}
\email{tuyentt@math.uio.no}
\address{Department of Mathematics, UiO, Norway.}

\input{commands.tex}

\begin{document}

\begin{abstract} 
A basic problem in the study of algebraic morphisms is to determine which sets can be realised as the image of an endomorphism of affine space. This paper extends the results previously obtained by the first author on the question of existence of surjective maps $F\colon \mathbb{A}^n \rightarrow \mathbb{A}^n\setminus Z$, where $Z$ is an algebraic subvariety of $\mathbb{A}^n$ of codimension at least 2. In particular, we show that for any (affine) algebraic variety $Z$ of dimension at most $n-2$, there is an algebraic variety $W\subset \mathbb{A}^n$ birational to $Z$ and a surjective algebraic morphism $\mathbb{A}^n\rightarrow \mathbb{A}^n\setminus W$. We also propose a conjectural approach towards resolving unknown cases.
\end{abstract}
\maketitle

\blfootnote{\emph{Key words}: affine space, endomorphism, Hilbert Nullstellensatz, surjectivity}

\section{Introduction} \label{sec: intro}
We work over an algebraically closed field $k$ of characteristic zero. Let $\mathbb{A}^n$ be the affine space of dimension $n$. A basic problem in the study of algebraic morphisms is the following: 

\begin{subtheo}\label{question: image of an to an}
    What subsets of $\mathbb{A}^n$ can be realised as the image of a dominant endomorphism $F\colon\mathbb{A}^n\rightarrow \mathbb{A}^n$?
\end{subtheo}

A classical general result is Chevalley's theorem 
\cite[Th\'{e}or\`{e}me 3]{chevalley1955schemas}
which asserts that the image of an algebraic map is constructible (see the appendix of this paper for some more detail on it). Aside from that, Question \ref{question: image of an to an} is largely open, except for an elementary argument that $\mathbb{A}^n\setminus F(\mathbb{A}^n)$ cannot contain a hypersurface (see e.g. \cite{arzhantsev2023images,kusakabe2022surjective}).

This paper concerns a nontrivial subproblem of Question \ref{question: image of an to an}, aiming to show the existence of and construct explicitly surjective morphisms $F\colon\mathbb{A}^n\rightarrow \mathbb{A}^n\setminus Z$, where $Z$ is a closed algebraic subvariety of codimension at least $2$. This question can naturally be stated as a stronger version of Hilbert's effective Nullstellensatz as follows. Let $(z_1,\ldots ,z_n)$ be the coordinates in the domain,  $(w_1,\ldots ,w_n)$ the coordinates in the range, $F=(F_1,\ldots ,F_n)$ the supposed endomorphism, and $Z=\{g_1=\ldots =g_m=0\}$. We want to check the two conditions: $F(\mathbb{A}^n)\subset \mathbb{A}^n\setminus Z$, and if $w\in \mathbb{A}^n\setminus Z$ then $F^{-1}(w)$ is non-empty. The first condition is equivalent to showing that $\{w_1-F_1=\ldots =w_n-F_n=0 = g_1=\ldots =g_m\}=\emptyset$, which can be interpreted (by Hilbert's Nullstellensatz) as the existence of polynomials $h_1,\ldots ,h_n, h_{n+1},\ldots ,h_{n+m}$ (in variables $z_1,\ldots ,z_n, w_1,\ldots, w_n$) so that $h_1(w_1-F_1)+\ldots h_n(w_n-F_n)+h_{n+1}g_1+\ldots +h_{n+m}g_m$ is identically $1$. The second condition is equivalent to showing that there do not exist  polynomials $h_1,\ldots ,h_n,h_{n+1}$ (in variables $z_1,\ldots ,z_n,t_1,\ldots ,t_m$) and $w_1,\ldots ,w_n$ so that $h_1(w_1-F_1)+\ldots +h_n(w_n-F_n)+h_{n+1}(t_1g_1+\ldots +t_mg_m-1)$ is identically $1$, where $t_1,\ldots ,t_m$ are new variables, which again can be reduced to Hilbert's Nullstellensatz. Note that $\mathbb{A}^n\setminus Z$ is a rational variety, and hence the existence of surjective maps $\mathbb{A}^n\rightarrow \mathbb{A}^n\setminus Z$ is interesting also from the viewpoint of birational geometry. So far, strongest evidences for the existence of such maps come from (algebraic) Oka theory. 

Over the complex numbers, such an $\mathbb{A}^n\setminus Z$ is a holomorphic (sub)elliptic manifold \cite{Gromov1989Okas}, and hence by \cite{Forstneric2017surjective} there are {\bf holomorphic} surjective maps $\mathbb{A}^n\rightarrow \mathbb{A}^n\setminus Z$. 
In the algebraic setting, however, this was previously known only for the case $Z\subset \mathbb{A}^{n-2}\subset \mathbb{A}^n$, by the work of the first author \cite{balch2023surjective} (and some earlier works on dimension $2$, \cite{Jelonek1999Anumberof, el2022counting}). On the other hand, \cite{kusakabe2022surjective}  showed that for the more general class of algebraic manifolds $X$ which are {\bf algebraically subelliptic} (see \cite{Gromov1989Okas} or \cite{Forstneric2017book} for a precise definition, and note the recent result that algebraic ellipticity and subellipticity are equivalent for smooth varieties \cite{KalimanZaidenberg+2023}) we always have surjective morphisms $\mathbb{A}^{\dim(X)+1}\rightarrow X$. (An earlier result \cite{arzhantsev2023images} showed that, for the special class of flexible manifolds $X$---which includes the manifolds of the form $\mathbb{A}^n\setminus Z$ considered in this paper---we always have surjective morphisms $\mathbb{A}^{N}\rightarrow X$,  where in general $N>\dim (X)$). For the case where $X$ is an algebraically (sub)elliptic {\bf compact}  manifold, Forstneri\v{c} showed the existence of a surjective algebraic map $\mathbb{A}^{\dim (X)}\rightarrow X$, \cite[Theorem 1.6]{Forstneric2017surjective}. See \cite{arzhantsev2023varieties} for many examples of algebraically (sub)elliptic compact manifolds.  Whether any algebraically (sub)elliptic  manifold $X$ always admits algebraic surjective morphisms $f\colon \mathbb{A}^{\dim X} \to X$ is an open question asked by several authors
\cite{Forstneric2017surjective,FORSTNERIC2023,Larusson2017Approx}, or indeed how we can characterize those varieties $X$ to which there is such a surjective map \cite[Problem 2]{arzhantsev2023images}. In the case $X=\mathbb{A}^n\setminus Z$, the results in this paper provide a solution to a weaker version of this question, see Corollary \ref{CorollaryBirationalProof}. 

A natural idea to construct surjective morphisms $\mathbb{A}^n\rightarrow \mathbb{A}^n\setminus Z$ is to start from a surjective morphism $g\colon\mathbb{A}^{N}\rightarrow \mathbb{A}^n\setminus Z$ (which exists by the above mentioned results \cite{arzhantsev2023images, kusakabe2022surjective}), then find an appropriate linear subspace $H\subset \mathbb{{A}}^{N}$ of dimension $n$ such that $F=g|_H$ is also surjective. In general, we may need to consider subvarieties $H$ of $ \mathbb{{A}}^{N}$ isomorphic to $\mathbb{A}^n$, and not just linear subspaces. One would expect that a generic choice of $H$ will not yield a surjective map $g|_H$, and one must therefore look for special subvarieties $H$. To this end, it will be very useful---at least in the first stages---to be able to efficiently test if the restrictions $g|_H$ to some reasonable choices of $H$ are surjective.

The main idea from \cite{kusakabe2022surjective} is as follows. If $X$ is a (sub)elliptic manifold of dimenison $n$, then for every $x\in X$ there is an algebraic morphism $F_x\colon\mathbb{A}^n\rightarrow X$---coming from the associated sprays, by using results from Forstneric \cite{Forstneric2006Holomorphicflex}---which is dominant and contains $x$ in its image. By Noetherian property, there is a finite number of such maps $F_{x_1},\ldots ,F_{x_N}\colon\mathbb{A}^n\rightarrow X$, such that the union of the images of $F_{x_1},\ldots ,F_{x_N}$ is $X$. We think about $F_{x_i}$ as a map from a hyperplane $H_i\subset\mathbb{A}^{n+1}$, where $H_i$ and $H_j$ are disjoint if $i\not= j$, and then extend these to a morphism $F\colon\mathbb{A}^{n+1} \rightarrow X$ whose image contains $\bigcup _{i=1}^NF_{x_i}(H_i)=X$.  

In the case where $X$ is of the form $\mathbb{A}^n\setminus Z$, where $Z$ is an algebraic subvariety of codimension at least 2,  a more convenient construction/procedure is given in \cite{arzhantsev2023images}, using the Gromov--Winkelmann theorem \cite[Proposition 1]{Winkelmann}, \cite[p. 72, Exercise (b')]{Gromov1986partial}. See Section \ref{sec: explicit computational} for more detail. This procedure yields an explicit surjective map, but typically with $N>n$. In our experiments, we find that there are cases where $N$ can be chosen to be $n+1$, but there are cases where it is not obvious so.

Our main result is Theorem \ref{thm: main}. 
\begin{thm} \label{thm: main}
Let $Z\subset \{0\} \times \A^{n-1} \subset \A^n$ be an algebraic subvariety. Then there exists a surjective algebraic map $F\colon\A^n\to \A^n\setminus Z$.
\end{thm}
In the proof we explicitly construct $F$. It is possible to do this in different ways. In Theorem \ref{thm: main thm with pure powers}, we give a construction which requires some mild assumptions on the choice of coordinate system, but achieves a map $F$ of lower degree.

Computational tools have been indispensable for us to find first non-trivial examples which in turn provided inspiration and ideas for the above mentioned theorems. The proofs of the above theorems are the generalizations of the proof we found first for a special case of dimension 3, by experimenting with the help of Macaulay2, see Example \ref{ex: ellcurve d=c} and Appendix \ref{subsec: M2 code for example} for detail. Theorem \ref{thm: main} is obtained by first constructing a surjective morphism $\mathbb{A}^{m+2n-2}\rightarrow \mathbb{A}^{n}\setminus Z$ (where $m$ is an upper bound for the number of equations defining  $Z$, and hence we can always choose $m\leq n$ by \cite{Kronecker,Storch1972,EisenbudEvans1973}), 
 and then restricting to an appropriate subvariety $H$ of $\mathbb{A}^{m+2n-2}$ which is isomorphic to $\mathbb{A}^n$. Theorem \ref{thm: main thm with pure powers} is obtained by first constructing a surjective morphism $\mathbb{A}^{n+m}\rightarrow \mathbb{A}^{n}\setminus Z$, and then restricting to an appropriate subvariety $H\cong \an$ of $\mathbb{A}^{n+m}$. 

We obtain the following consequence, which can be viewed as almost a solution to the question about the existence of such surjective morphisms $\mathbb{A}^n\rightarrow \mathbb{A}^n\setminus Z$.  

\begin{cor}
    Let $Z$ be an (affine) algebraic variety of dimension at most $n-2$. Then there are an algebraic variety $W\subset \mathbb{A}^n$ birational to $Z$ and a surjective morphism $F\colon\mathbb{A}^n\rightarrow \mathbb{A}^{n}\setminus W$. 
\label{CorollaryBirationalProof}\end{cor}
\begin{proof}
    It is a classical result that there is an algebraic variety $W\subset \{0\}\times\mathbb{A}^{n-1}\subset\mathbb{A}^n$ birational to $Z$ (see \cite[Theorem 1.8]{Shafarevich2013alggeo1}, \cite[Theorem I.4.9]{hartshorne1977}).  
    Theorem \ref{thm: main} furnishes a surjective algebraic morphism $F\colon\mathbb{A}^n\rightarrow \mathbb{A}^n\setminus W$. 
\end{proof}

Given a variety $X$, we denote by $\aut(X)$ its group of automorphisms. A nontrivial regular action $\varphi\colon \ga \times X \to X$ by the additive group $\ga = (k, +)$ gives rise to a subgroup $G$ of $\aut(X)$, which we call a \emph{$\ga$-subgroup}.
We consider $X = \an\setminus Z$ and fix some point $p\in \an\setminus Z$. By \cite[Proposition 2, Corollary 1]{arzhantsev2023images} (see also \cite[Corollary 1.11]{arzhantsev2013flexible}), one can find a finite sequence of $\ga$-subgroups $G_i$ of $\aut (\an\setminus Z)$ such that
\begin{equation*}
    (G_N\cdot \ldots \cdot G_1).p = \an\setminus Z.
\end{equation*}
The sequence of groups $G_1,\ldots,G_N$ furnishes a surjective morphism $F\colon\A^N \to \an\setminus Z$, where $F\colon (g_1,g_2,\ldots,g_N) \mapsto (g_N \ldots g_1.p)$. 

If $W$ is an algebraic subvariety of $\mathbb{A}^n$, we denote by $\aut(\mathbb{A}^n,W)$ the set of algebraic automorphisms of $\mathbb{A}^n$ fixing $W$ pointwise. 
Towards a complete solution for the existence of surjective algebraic maps $\mathbb{A}^n\rightarrow \mathbb{A}^n\setminus Z$, we propose the following conjectural approach: 

\begin{conjecture}\label{ConjectureApproach}
     Let $Z\subset \mathbb{A}^n$ be an algebraic subvariety of codimension at least $2$. Then there exist 
     
     1) $N$ hypersurfaces $W_1,\ldots ,W_{N}\subset \mathbb{A}^n$ containing $Z$, with $\bigcap _iW_i = Z$, 
     
    2) $\ga$-subgroups $G_{i}\subset \aut(\mathbb{A}^n,W_i)$,

     3) A point $p\in \mathbb{A}^n\setminus Z$, 
  and
     
     4) A subvariety $H\subset \mathbb{A}^{N}$ isomorphic to $\mathbb{A}^{n}$, 
     
     such that for $\psi:=(G_{N}\cdot\ldots\cdot G_{1}).p\colon \mathbb{A}^{N}\rightarrow \mathbb{A}^n\setminus Z$, the restricted map $\psi|_H$ is surjective.
\end{conjecture}

In the setting of Conjecture \ref{ConjectureApproach}, since $Z\subset W_i$, $\aut(\mathbb{A}^n,W_i)$ is naturally a subgroup of $\aut(\mathbb{A}^n\setminus Z)$, for all $i$. We will sometimes write $G_{i,a_i} = G_i$ to emphazise the parameter $a_i$. While similar to the setting in \cite{arzhantsev2023images}, which looks at 1-parameter subgroups of $\aut(\mathbb{A}^n\setminus Z)$, the above formulation of Conjecture \ref{ConjectureApproach} aims for a more concrete strategy to look for surjective morphisms $\mathbb{A}^n\rightarrow \mathbb{A}^n\setminus Z$ in practical settings.

Both proofs of Theorems \ref{thm: main} and \ref{thm: main thm with pure powers} follow the strategy in Conjecture \ref{ConjectureApproach}.  

{\bf Plan of the paper.} In Section \ref{sec: proofs} we prove the main theoretical results. In Section \ref{sec: explicit computational} we present some computational experiments, which are the inspiration for Theorems \ref{thm: main} and \ref{thm: main thm with pure powers} and are evidences to support Conjecture \ref{ConjectureApproach}. In the appendix we present some algorithms used in the experiments, and illustrations of how one can write code in Macaulay2 to perform the experiments. 

{\bf Acknowledgements.} The authors are partially supported by the Young Research Talents grant 300814 from Research Council of Norway. The first author thanks Mateusz Micha{\l}ek for some help with the code of the paper \cite{Harris2019images}, and thanks Finnur L\'arusson for some discussions on the problem while at University of Adelaide several years ago. The authors thank John Christian Ottem for interest and useful comments.

\section{Proofs of main results} \label{sec: proofs}

We now prove our main theorem, Theorem \ref{thm: main}. Then we state and prove an alternate version, giving a map of lower degree, Theorem \ref{thm: main thm with pure powers}.

\begin{proof}[Proof of Theorem \ref{thm: main}]
We begin by constructing a surjective map $\psi\colon \A^{m+2n-2}\to \A^n\setminus Z$ by using group actions of $\ga$-subgroups of $\aut(\an\setminus Z)$. 
$Z$ can be written as $Z=V(w_1, q_1,\ldots, q_m)$, where each $q_i$ is a polynomial in $w_2,\ldots,w_n$.
We use $2n-2+m$ groups $G_i$, and write $\varphi_i$ for the group action $\varphi_i\colon G_i\times \an\setminus Z \to \an\setminus Z$. We use the following $\varphi_i$'s. 
\begin{align*}
\varphi_i\colon (a_i, w_1,\ldots,w_n) &\mapsto 
(w_1,\ldots, w_i, w_{i+1}+a_iw_1, w_{i+2},\ldots,w_n)\\
\varphi_{i+n-1}\colon (c_i, w_1,\ldots,w_n) &\mapsto 
(w_1+c_iq_i(w_2,\ldots,w_n), w_2,\ldots,w_n)\\
\varphi_{i+n-1+m}\colon (b_i, w_1,\ldots,w_n) &\mapsto 
(w_1,\ldots, w_i, w_{i+1}+b_iw_1, w_{i+2},\ldots,w_n).
\end{align*}
We fix a point $p=(1,0,\ldots,0)$ and compose the group actions to obtain $\psi = (G_{2n-2+m}\ldots G_1).(p)\colon \A^{2n-2+m}\to \A^n\setminus Z$, with the following explicit formula.
\begin{multline} \label{eq: surjective psi in ai ci bi}
\psi = (1+\sum_{j=1}^{j=m}
    c_jq_j(a_1,\ldots,a_{n-1}),
    a_1+b_1w_1,\ldots,a_{n-1}+b_{n-1}w_1), 
\end{multline}
where $w_1$ is shorthand for the expression in the first coordinate,  $w_1= 1+\sum_{j=1}^{j=m} c_jq_j.$ We now find an $\an \subset \A^{m+2n-2}$ such that the restriction of $\psi$ to $\an$ is still surjective.

Let $d= \max(\deg(q_j))$, and set $p_k := d^k+d^{k-1}+\ldots+1$. The subvariety $H$ defined by $c_j=c_1^{(j-1)(1+p_{n-1})}$ for all $j\neq 1$, and  $b_i= c_1^{p_{i-1}}$ for all $i$, i.e. $H = V(c_j-c_1^{(j-1)(1+p_{n-1})}, b_i- c_1^{p_{i-1}})$, is isomorphic to $\mathbb{A}^n$. We define $F= \psi|_H$ and will prove that $F(a_i, c_1)$ is surjective. Below, we write $c$ in place of $c_1$.

Let $F\colon\A^n\to\A^n\setminus Z$ be
\begin{multline}
    F(a_i,c)= (1+cq_1+\sum_{j=2}^{j=m}
    c^{(j-1)(1+p_{n-1})}q_j(a_1,\ldots,a_{n-1}),\\
    a_1+cw_1,a_2+c^{d+1}w_1,\ldots, a_{i-1}+c^{p_{i-2}}w_1,\ldots,a_{n-1}+c^{p_{n-2}}w_1),
\end{multline}
where again, $w_1$ is shorthand for $w_1 = 1+cq_1+\sum_{j=2}^{j=m}c^{(j-1)(1+p_{n-1})}q_j(a_1,\ldots,a_{n-1})$. 
We check whether a point $(w_1,\ldots,w_n)$ is in the image, by checking whether the preimage $F\inv(w_1,\ldots,w_n)$ is empty. 
We have $n$ equations:
\begin{align}
    w_1&=1+cq_1+c^{1+p_{n-1}}q_2 +\ldots +c^{1+(m-1)p_{n-1}}q_m & w_i&= a_{i-1}+c^{p_{i-2}}w_1.
\end{align}
The last $n-1$ equations are solved for $a_i$ to obtain $a_{i-1}=w_i-c^{p_{i-2}}w_1$. It only remains to solve the following polynomial equation in $c$:
\begin{align} \label{eqn: poly in c}
    0&=cq_1+c^{1+p_{n-1}}q_2 +\ldots +c^{1+(m-1)p_{n-1}}q_m-w_1+1,
\end{align}
where now $q_j = q_j(a_1,\ldots,a_{n-1}) = q_j(w_2-cw_1, \ldots, w_n-c^{p_{n-2}}w_1).$

This has no solutions if and only if for all $j$, $q_j(w_2-cw_1, \ldots, w_n-c^{p_{n-2}}w_1)\equiv 0$.
This happens when $w_1=0$ and $q_j(w_2,\ldots,w_n)=0$ for all $j$. If any of those conditions are not met, then Eq. (\ref{eqn: poly in c}) is non-constant polynomial, and has solutions in $c$, since $k$ is algebraically closed. We have now computed that the preimage of a point $(w_1,\ldots,w_n)$ is empty if and only if $(w_1,\ldots,w_n)\in Z$, and conclude that $F\colon\an\to\an\setminus Z$ is surjective.
\end{proof}
\begin{rem}
An upper bound for the degree of $F$ constructed in the above proof is $mp_{n-1} = md^{n-1}+md^{n-2}+\ldots+m$.
\end{rem}

We now state and prove an alternate version of our main theorem, which can provide maps of smaller degrees. In order to do that, we first need the following lemma, which is probably well known.

\begin{lem}\label{lem: power of y_n}
    For any set of nonconstant polynomials $\{q_j(x_1,\ldots,x_n)\}$, where each $q_j$ has degree $d_j$, we can always find a linear change of coordinates from $x_i$ to $y_i$ in such a way that each $q_j$ contains a monomial $a_jy_n^{d_j}$ for some constants $a_j\not=0$. 
\end{lem}
\begin{proof} One can define new variables $y_i=x_i+A_ix_n$, $(1\leq i\leq n-1)$, and $y_n=x_n$. For a general choice of $A_i$'s, the conclusion of the Lemma is satisfied. Indeed, any monomial of degree $d$ in the $x_i$'s will be transformed by this linear map to a homogeneous polynomial in the $y_i$'s which contains a monomial $ay_n^d$. We just need to choose the numbers $A_i$'s so that the coefficient of $y_n^{d_j}$ (which is a polynimial in the $A_i$'s) is non-zero.   
\end{proof}

\begin{rem}
    What Lemma \ref{lem: power of y_n} achieves is that for any fixed tuple $(y_1,\ldots,y_{n-1})$, one can change the value of $q_j(y_1,\ldots,y_n)$ by changing $y_n$.
\end{rem}

\begin{thm}
\label{thm: main thm with pure powers}
Let $Z$ be a subvariety $Z\subset \left(\{0\}\times \A^{n-1} \right) \subset \A^n$, defined by $Z=V(w_1, q_1(w_2,\ldots,w_n),\ldots, q_m(w_2,\ldots,w_n))$, where each $q_j$ contains a monomial $a_jw_n^{d_j}$ in the sense of Lemma \ref{lem: power of y_n}. 
Then there exists a surjective function $F\colon\A^n\to \A^n\setminus Z$ of degree $\leq m(d+1)+1$, where $d= \max(\deg(q_j))$.
\end{thm}

\begin{proof}
By the same procedure as in the proof of \ref{thm: main}, we construct the map $\psi$ in Eq. (\ref{eq: surjective psi in ai ci bi}). However, we then restrict to the subvariety defined by $c_j=c_1^{(j-1)(d+1)+1}$ for all $j$, $b_i = 0$ for all $i< n-1$, and $b_{n-1}=c_1$, and obtain the following map.
\begin{align}
    F(a_i,c)&= (1+
    \sum_{j=1}^{j=m}c^{(j-1)(d+1)+1}q_j,
    a_1,a_2,\ldots,a_{n-1}+cw_1),
\end{align}
where $q_j=q_j(a_1,\ldots,a_{n-1})$. 
Computing the preimage of a point $w$ reduces to solving the following 
polynomial equation in $c$:
\begin{align} \label{eq: poly in c thm pure}
    0&=
    cq_1+\ldots+c^{(m-1)(d+1)+1}q_m 
    -w_1+1,
\end{align}
where $q_j=q_j(w_2, \ldots, w_n-cw_1)$.
This has no solutions if and only if for all $j$, $q_j(w_2, w_3,\ldots, w_n-cw_1)\equiv 0$.
This happens when $w_1=0$ and all $q_j(w_2,\ldots,w_n)=0$. Lemma \ref{lem: power of y_n} ensures that Eq. \ref{eq: poly in c thm pure} is solvable otherwise.
\end{proof}

\section{Some explicit computational examples} \label{sec: explicit computational}
In this section we present some experiments which are the inspiration for Theorems \ref{thm: main} and \ref{thm: main thm with pure powers}, as well as give support to Conjecture \ref{ConjectureApproach}. 

To be able to perform the computations, we need the following tools: 

Tool 1: Algorithms to compute the image of algebraic morphisms. We will in particularly use the ideas from \cite{Barakat2021algorithmic}. For the readers' convenience, we briefly recall this in the appendix.   

Tool 2: Algorithms to construct surjective morphisms $\mathbb{A}^N \rightarrow \an \setminus Z$, where in general $N$ can be much bigger than $n$. 

In Section \ref{sec: intro}, we mentioned the general result by \cite{kusakabe2022surjective} on the existence of surjective algebraic maps $\mathbb{A}^{n+1}\rightarrow X$, for algebraic subelliptic manifolds $X$ of dimension $n$.  However, this construction (because of some non-intrinsic features) may yield complicated maps which make computations inefficient in the cases we consider. 

Instead, here we use the ideas from \cite{Winkelmann}, and subsequent developments in \cite{arzhantsev2013flexible, arzhantsev2023images} for flexible varieties. Here we briefly recall the idea for our setting $X=\A^n\setminus Z$. We can start with $\ga$-subgroups  $G_1,\ldots ,G_n$ of $\aut(X)$ (constructed explicitly in \cite{Winkelmann} by utilising general linear projections to hyperplanes of $\A^n$), so that the orbit of one point by $F_n=(G_n\cdot \ldots \cdot G_1)$ contains a Zariski-open subset in $X$. Then, for any point $x_1$ in $X$ not contained in the image, we compose $F_n$ with the action of another 1-parameter subgroup $G_{n+1}(a_{n+1},.)$ to cover $x_1$.  Since the action $G_i(0,.)$ is the identity map on $X$, it follows that the image of $F_{n+1}=(G_{n+1} \cdot F_n).x_0$ is strictly bigger than that of $F_n$ and contains $x_1$. Composing with more such maps, we can even arrange that the map is dominant over $x_1$. We inductively do this whenever there are still some missing points. This process will terminate after a finite number of steps, thanks to Chevalley's theorem and the Noetherian property.

In \cite{Winkelmann}, an explicit way to construct such 1-parameter families of automorphisms for $X=\mathbb{A}^n\setminus Z$ is given. It works as follows. For a generic coordinate system $z_1,\ldots ,z_n$, we consider the projection $\pi _i$ from $\mathbb{A}^n$ to the hyperplane $\{z_i=0\}$, and let $f_i(z_1,\ldots ,z_{i-1},z_{i+1},\ldots ,z_n) $ be a polynomial which vanishes on $\pi _i(Z)$. Then we construct the $\ga$-subgroup $G_i(a_i,.)\in \aut(\mathbb{A}^n,\{f_i=0\})$  by the following formula: 
\begin{equation}
G_i(a_i,z)=(z_1,\ldots ,z_{i-1}, z_i+a_if_i(z_1,\ldots ,z_{i-1},z_{i+1},\ldots ,z_n),z_{i+1},\ldots ,z_n).
\end{equation}

We write $W_i = \{f_i=0\}$, for the fixed point set of $G_i$.
In experiments we find that there are cases (even in small dimensions like $3$) the construction by Winkelman \cite{Winkelmann} may yield very complicated polynomials making computations with Macaulay2 unfeasible. In such cases, we may need to work with other 1-parameter families of automorphisms. Conjecture \ref{ConjectureApproach} was stated to cover such cases.

Now we are ready to present our experiments, following the strategy proposed in Conjecture \ref{ConjectureApproach}.  
This first example inspired Theorems \ref{thm: main} and \ref{thm: main thm with pure powers}.

\begin{exam}[Smooth cubic curve] \label{ex: ellcurve d=c} Let $E=V(w_2^2-w_3^3-w_3,w_1)\subset \mathbb{A}^3$ be a smooth plane cubic curve. We want to construct a surjective morphism $\mathbb{A}^3\rightarrow \mathbb{A}^3\setminus E$. 

According to the general strategy, we find $\mathbb{G}_a$-actions on $\A^3$ that fix the curve $E$. 
By using the idea of projection to hyperplanes from \cite{Winkelmann}, we find actions $\varphi_i\colon G_i\times \mathbb{A}^3\setminus E \to \mathbb{A}^3\setminus E$.
\begin{align*}
    \varphi_1: (t,w) &\mapsto (w_1,w_2+tw_1,w_3)\\
    \varphi_2: (t,w) &\mapsto (w_1,w_2,w_3+tw_1)\\
    \varphi_3: (t,w) &\mapsto (w_1+t(w_2^2-w_3^3-w_3),w_2,w_3).
\end{align*}
Set $\varphi_4=\varphi_2$, and
note that the fixed hypersurfaces are $W_1=V(w_2^2-w_3^3-w_3)$ and $W_2=W_3=W_4=V(w_1)$. 
We pick $p=(1,0,0)\in \mathbb{A}^3\setminus E$ and composing group actions $(G_4\cdot G_3 \cdot G_2 \cdot G_1).p $ to obtain a surjective map to $\A^3\setminus E$.
\begin{align*}
    \varphi_1: (a,1,0,0) &\mapsto (1,a,0)\\
    \varphi_2: (b,1,a,0) &\mapsto (1,a,b)\\
    \varphi_3: (c,1,a,b) &\mapsto (1+c(a^2-b^3-b),a,b)\\
    \varphi_4: (d,1+c(a^2-b^3-b),a,b) &\mapsto (1+c(a^2-b^3-b),a,b+d(1+c(a^2-b^3-b))).
\end{align*}
The resulting map $\psi\colon \A^4_{a,b,c,d} \to \A_{w_1,w_2,w_3}^3\setminus E$ is given by 
\begin{align}
    \psi(a,b,c,d) &\mapsto (1+c(a^2-b^3-b),a,b+d+cd(a^2-b^3-b)).
\end{align}
We claim $\psi$ is surjective. 
Moreover, the restriction $F=\psi|_{d=c}$ is surjective, and hence we obtain a surjective morphism $\colon \mathbb{A}^3\rightarrow \mathbb{A}^3\setminus E$. This is checked with the Macaulay2 code in Appendix \ref{subsec: M2 code for example}, and can also be seen directly by using the proof of Theorem \ref{thm: main thm with pure powers}.
\end{exam}

The next example is not literally covered by Theorems \ref{thm: main} and \ref{thm: main thm with pure powers}, but can be covered by these theorems after a change of coordinates. 

\begin{exam}[Twisted cubic]
Let $A$ be the curve $A=V(w_2^2-w_1,w_2^3-w_3).$ The approach of constructing $\ga$-actions by projection to coordinate hyperplanes as in \cite{Winkelmann} does not obviously result in a surjective map $\A^4\to \A^3\setminus A$.

Instead, we may change coordinates appropriately to reduce to the case where Theorems \ref{thm: main} and \ref{thm: main thm with pure powers} apply. Let $\tau: \A^3\to \A^3$ be the automorphism 
\begin{equation*}
     \tau: (w_1,w_2,w_3) \mapsto (w_1-w_2^2,w_2,w_3-w_2^3),
\end{equation*}
and its inverse is $\tau\inv = (w_1+w_2^2,w_2,w_3+w_2^3)$.
Notice that $\tau(A)=(0,w_2,0)$.
We now use Theorem \ref{thm: main thm with pure powers} with $q_1=w_3$. This yields
\begin{equation}\label{eq: f line in c3}
    f({a},{b},c)= (1+c{b},
    {a},{b}+c(1+c{b})),
\end{equation}
which has image $\A^3\setminus V(w_2)$. Hence $F=\tau\inv \circ f$ has image $\A^3\setminus A$.

We may also arrive at $F$ as follows. For each group action $\varphi_i\colon G_i \times X \to X$ used to build $f$, we define $\tilde\varphi_i = \tau\inv(\varphi_i(G_i,\tau(X)))$. In this example
\begin{align}
    \varphi_1(t) &= (w_1+tw_3,w_2,w_3)&
    \tilde\varphi_1 &= (w_1+t(w_3-w_2^3),w_2,w_3)\\
    \varphi_3(t) &= (w_1,w_2,w_3+tw_1)&
    \tilde\varphi_3 &= (w_1,w_2,w_3+t(w_1-w_2^2))\\
    \varphi_2(t) &= (w_1,w_2+tw_1,w_3).&&
\end{align}
Finally, we have
\begin{multline}
        \tilde\varphi_2 = (w_1-w_2^2+ (w_2+{a}(w_1-w_2^2))^2, w_2+{a}(w_1-w_2^2), \\
        w_3-w_2^3+(w_2+{a}(w_1-w_2^2))^3)
\end{multline}
Note that the $\tilde\varphi_i$'s respectively fix the subvarieties $\tilde{W}_1 = V(w_3-w_2^3)$, and $\tilde{W}_2= \tilde{W}_3 = V(w_1-w_2^2)$. We pick a point $p=(1,0,0)$ and act (by slight abuse of notation) by $\psi = (\tilde\varphi_3\cdot \tilde\varphi_1 \cdot \tilde\varphi_3 \cdot \tilde\varphi_2).p : \A^4 \to \A^3\setminus A$. Then $F \defeq \psi|_{c=d}$ is given by 
\begin{equation} \label{eq: F twisted cubic}
    F({a},{b},c)= (1+c{b}+{a}^2,
    {a},{b}+c(1+c{b})+{a}^3),
\end{equation}
which has image $\A^3\setminus A$. We can also use Macaulay2, similar to the code in the appendix, to verify that indeed we have a surjective map $\A^3\rightarrow \A^3\setminus A$.
\end{exam}

\bibliography{mylib}

\newpage

\appendices
\section{Algorithmically computing the image}
In this appendix we briefly summarize an algorithm to compute the image of an algebraic morphism which is convenient for our setting, and illustrate how we can write code in Macaulay2 to run the experiments in Section \ref{sec: explicit computational}.

\subsection{An algorithm for computing the image of an algebraic map} \label{sec: algorithm to check surjectivity}
We recall here the algorithm in \cite{Barakat2021algorithmic} for computing the constructible image of an algebraic map, which has been convenient for our purposes. It is also worth mentioning the Macaulay2-library TotalImage described in \cite{Harris2019images}, which has been very useful for maps of lower degree.

Given a polynomial map $f\colon Z\rightarrow \an$, where $Z\subset \mathbb{A}^N$ is a closed algebraic subvariety. We want to decide $f(Z)$, which is a constructible set by Chevalley's theorem. 

We can assume that $Z$ is irreducible. We note that $\overline{f(Z)}$ is a closed algebraic subvariety of $\an$. The main idea is to find a {\bf proper} algebraic subvariety $W\subset \overline{f(Z)}$ so that $\overline{f(Z)}\setminus f(Z)\subset W$. Then, we have: 
\begin{eqnarray*}
    f(Z)=[\overline{f(Z)}\setminus W]\sqcup [W\cap f(Z)]=[\overline{f(Z)}\setminus W]\sqcup [f(f^{-1}(W)\cap Z)].
\end{eqnarray*}

Since $f^{-1}(W)\cap Z$ is a closed algebraic subvariety of $\mathbb{A}^N$ of smaller dimension than $Z$, we can proceed inductionaly to have a complete decomposition of $f(Z)$ into sets of the form $V_1\setminus V_2$, where both $V_1$ and $V_2$ are closed algebraic subvarities of $\mathbb{A}^n$. 

We can first reduce to the case where the dimensions of $Z$ and $f(Z)$ are the same. If not, then we can choose a random linear subspace $L\subset\mathbb{A}^N$ of appropriate dimension such that $L\cap Z$ has the same dimension as $f(Z)$ and $f(L\cap Z)$. (This is the same as checking that $\overline{f(Z)}$ has the same dimension as $\overline{f(L\cap Z)}$ and $L\cap Z$.)

With the assumption as in the previous paragraph, we can explicitly find one such $W$ as follows. We let $\Gamma _f\subset \mathbb{A}^N\times \an$ be the graph of $f$, let $H=\mathbb{P}^N\setminus \mathbb{A}^N$ be the hyperplane at infinity, $\overline{\Gamma}_f\subset \mathbb{P}^N\times \an$ be the closure of $\Gamma _f$, and $\Gamma _f^{\infty}=\overline{\Gamma}_f\cap (H\times \an)$. Then $W$ can be chosen as the projection of $\Gamma _f^{\infty}$ to the second factor $\an$. 

The above algorithm gives a constructive proof of Chevalley's theorem mentioned before. 

Implementation detail: The above algorithm can be implemented in different computer algebra systems. In our experiments later, we will use the free software Macaulay2 \cite{M2}.

We note that the map $f$ can be identified with the projection to $\an$ from the graph $\Gamma _f$. As such, the closure of the image of $f$ can be computed by the elimination technique, using Gr\"obner basis.

The closure $\overline{\Gamma}_f\subset \mathbb{P}^N\times \an$ can be computed by homogenizing, then saturating with respect to the ideal of the hyperplane at infinity. 

The closure of the image of the projection to $\an$ from $\Gamma _f^{\infty}$ again can be computed by the elimination of variables.

\subsection{Illustrating code in Macaulay2}  
\label{subsec: M2 code for example}
Here we present the code in Macaulay2 used to verify Example \ref{ex: ellcurve d=c}. It executes the algorithm described above. Note that it is not a complete implementation of the algorithm in \cite{Barakat2021algorithmic}. 

\begin{lstlisting}[language=Macaulay2]
r = QQ[a,b,c,e,w1,w2,w3, Degrees => {1,1,1,1,0,0,0}]
 affineGraph = ideal(1+c*(a^2-b^3-b)-w1,a-w2,b+c+c^2*(a^2-b^3-b)-w3)
 homogenized = homogenize(affineGraph,e)
 graphClosure = saturate(homogenized, e)
 graphAtInfty = graphClosure + ideal(e)
\end{lstlisting}
We have now computed \texttt{graphAtInfty}, which is  $\Gamma_f^{\infty} \subset \mathbb{P}^{2}\times \A^3$. We now compute its projection down to $\A^3$ from the different open affine subsets by elimination.
\begin{lstlisting}[language=Macaulay2]
 h1  = graphAtInfty + ideal(a-1)
 h1elim = eliminate({a,b,c,e}, h1)
 h2 = graphAtInfty + ideal(b-1)
 h2elim = eliminate({a,b,c,e}, h2)
 h3 = graphAtInfty + ideal(c-1)
 h3elim = eliminate({a,b,c,e}, h3)
\end{lstlisting}
After eliminating $a,b,c,e$ from \texttt{h1} and \texttt{h2}, the resulting ideals are just $(1)$. With \texttt{h3} on the other hand, we get a nonzero ideal, whose radical is $(w_1)$. We now add this to our original graph ideal to get $\Gamma_1$.
\begin{lstlisting}[language=Macaulay2]
 idealToAdd = radical h3elim
 gamma1 = affineGraph + idealToAdd
\end{lstlisting}
Then we repeat the whole process replacing \texttt{affineGraph} with \texttt{gamma1}.
\begin{lstlisting}[language=Macaulay2]
 homogenized = homogenize(gamma1,e)
 graphClosure = saturate(homogenized, e)
 graphAtInfty = graphClosure + ideal(e)
 h1  = graphAtInfty + ideal(a-1)
 h1elim = eliminate({a,b,c,e}, h1)
 h2 = graphAtInfty + ideal(b-1)
 h2elim = eliminate({a,b,c,e}, h2)
 h3 = graphAtInfty + ideal(c-1)
 h3elim = eliminate({a,b,c,e}, h3)
 idealToAdd = radical h3elim
 gamma2 = gamma1 + idealToAdd
 homogenized = homogenize(gamma2,e)
\end{lstlisting}
This final homogenized ideal is $(1)$, meaning that the algorithm finished, and that the added ideal $(w_1, w_2^2-w_3^3-w_3)$ describes the complement of the image, $Z = V(w_1,w_2^2-w_3^3-w_3)$.

\end{document}

%% file: commands.tex
\usepackage{hyperref}
\usepackage{setspace}
\usepackage{graphicx}
\usepackage{amssymb}
\usepackage{mathrsfs}
\usepackage{amsthm}
\usepackage{amsmath}
\usepackage{color}
\usepackage[font=small,labelfont=bf]{caption}
\usepackage{fancyhdr}
\usepackage{times}
\usepackage[capitalize]{cleveref}
\usepackage{float}
\usepackage{extarrows}
\restylefloat{figure}
\usepackage{algorithm}
\usepackage{algpseudocode}

\usepackage{mathtools}
\usepackage{faktor}
\usepackage{amsfonts}
\usepackage{color}
\usepackage{graphicx}
\usepackage{xspace}
\usepackage{blindtext}
\usepackage{url}
\usepackage{listings}
\lstdefinestyle{mystyle}{
    basicstyle=\ttfamily\footnotesize,
    tabsize=2
}
\usepackage{graphicx}
\usepackage{float}
\usepackage{xcolor}
\usepackage{etoolbox}
\usepackage{caption}
\usepackage{xpatch}
\usepackage{appendix}
\usepackage{parskip}
\usepackage{tikz}
\usepackage{tikz-cd}
\usepackage{enumerate}
\usepackage{pgfplots}
\pgfplotsset{compat=newest}
\usepackage{wrapfig}
\usepackage{stmaryrd}
\usepackage{comment}
\usepackage{tikz}
\usepackage{pgfplots}
\usepackage{tikz-3dplot}
\tdplotsetmaincoords{60}{115}
\pgfplotsset{compat=newest}
\usetikzlibrary{calc}
\usepackage{subcaption}
\usepackage{comment}

\newcommand\blfootnote[1]{%
  \begingroup
  \renewcommand\thefootnote{}\footnote{#1}%
  \addtocounter{footnote}{-1}%
  \endgroup
}

\usepackage{datetime}

\newdateformat{monthyeardate}{%
  \monthname[\THEMONTH], \THEYEAR}

\definecolor{NTNUblue}{RGB}{0,80,158}
\definecolor{NTNUbluesupport}{RGB}{62,98,138}
\definecolor{NTNUorange}{RGB}{239,129,20}
\definecolor{NTNUpurple}{RGB}{176,27,129}
\definecolor{NTNUturquoise}{RGB}{68,186,190}
\definecolor{NTNUbeige}{RGB}{207,184,135}

\pgfplotsset{compat=1.15}

\xspaceaddexceptions{]\}}
\newcommand{\fix}[1]{\ifmmode{#1}\else{$#1$}\xspace\fi}
\newcommand{\A}{\mathbb{A}}
\newcommand{\an}{\mathbb{A}^{n}}

\newcommand{\ga}{\mathbb{G}_a}

\newcommand{\defeq}{\vcentcolon =}

\newcommand{\inv}{^{-1}}

\renewcommand{\phi}{\varphi}
\renewcommand{\epsilon}{\varepsilon}

\renewcommand{\and}{\quad \text{ and } \quad}

\newcommand*{\doublerightarrow}[2]{\mathrel{
  \settowidth{\@tempdima}{$\scriptstyle#1$}
  \settowidth{\@tempdimb}{$\scriptstyle#2$}
  \ifdim\@tempdimb>\@tempdima \@tempdima=\@tempdimb\fi
  \mathop{\vcenter{
    \offinterlineskip\ialign{\hbox to\dimexpr\@tempdima+1em{##}\cr
    \rightarrowfill\cr\noalign{\kern.5ex}
    \rightarrowfill\cr}}}\limits^{\!#1}_{\!#2}}}

\newcommand\restr[2]{{
  \left.\kern-\nulldelimiterspace 
  #1 
  \vphantom{\big|} 
  \right|_{#2} 
  }}

\DeclareMathOperator{\aut}{Aut}

\theoremstyle{plain}
\newtheorem{thm}{Theorem}[section]
\newtheorem{cor}[thm]{Corollary}

\newtheorem{lem}[thm]{Lemma}

\newtheorem{conjecture}[thm]{Conjecture}

\theoremstyle{definition}
\newtheorem{test}{theorem}[section]

\newtheorem{subtheo}{Question}[test] 

\newtheorem{exam}[thm]{Example}

\newtheorem{rem}[thm]{Remark}